\documentclass[final,leqno]{siamltex704}

\usepackage{amsmath}
\usepackage{amssymb}
\usepackage{graphicx}
\usepackage{float}
\usepackage{comment}
\usepackage{xcolor}
\usepackage{booktabs,etoolbox}
\usepackage{tikz}
\usepackage{enumerate}
\usepackage{hyperref}

\newtheorem{example}{\bf Example}[section]
\newtheorem{algorithm}{Stabilizer free Weak Galerkin Algorithm}
\newtheorem{Corollary}{Corollary}[section]

\def\3bar{{|\hspace{-.02in}|\hspace{-.02in}|}}

\title{A note on the optimal degree of the weak gradient of the  stabilizer free weak Galerkin finite element method}
\author{Ahmed Al-Taweel\thanks{Department of Mathematics, University of Arkansas at Little Rock, Little Rock, AR 72204 (email: \href{mailto:Ahmed Al-Taweel@ualr.edu}{asaltaweel@ualr.edu})}\and Xiaoshen Wang\thanks{Department of Mathematics, University of Arkansas at Little Rock, Little Rock, AR 72204 (email:\href{mailto:Xiaoshen Wang@ualr.edu}{xxwang@ualr.edu} )}}
\begin{document}
\maketitle

\begin{abstract}
Recently, a new stabilizer free weak Galerkin method (SFWG) is proposed, which is easier to implement. The idea is to raise the degree of polynomials $j$ for computing weak gradient. It is shown that if $j\geq j_{0}$ for some $j_{0}$, then SFWG  achieves the optimal rate of convergence. However, large $j$ will cause some numerical difficulties. To improve the efficiency of SFWG and avoid numerical locking, in this note, we provide the optimal $j_{0}$ with rigorous mathematical proof.
\end{abstract}
\begin{keywords}
weak Galerkin finite element methods, weak gradient, Error estimates.
\end{keywords}
\begin{AMS}
Primary: 65N15, 65N30; Secondary: 35J50
\end{AMS}
\pagestyle{myheadings}
\section{Introduction}
In this note, similar to \cite{b2}, we will consider the following Poisson equation
\begin{eqnarray}
-\Delta u &=& f \quad \mbox{ in }\Omega, \label{bvp}  \\
u&=&0 \quad \mbox{ on }\partial\Omega, \label{dbc}
\end{eqnarray}
as the model problem, where $\Omega$ is a polygonal domain in $\mathbb{R}^2$.
The variational formulation of the Poisson problem (\ref{bvp})-(\ref{dbc}) is to seek $u\in H_{0}^{1}(\Omega)$ such that
\begin{eqnarray}
(\nabla u,\nabla v)&=&(f,v),\qquad \forall v\in H_{0}^{1}(\Omega).\label{weak form}
\end{eqnarray}

A stabilizer free  weak Galerkin finite element method is proposed by Ye and Zhang in \cite{b2} as a new method for the solution of Poisson equation on polytopal meshes in 2D or 3D. Let $j$ be the degree of polynomials for the calculation of weak gradient. It has been proved in \cite{b2} that there is a $j_{0}>0$ so that the SFWG method converges with optimal order of convergence for any $j\geq j_{0}$. However, when $j$ is too large, the magnitude of the weak gradient  can be extremely large, causing numerical instability. In this note, we provide the optimal $j_{0}$ to improve the efficiency and avoid unnecessary numerical difficulties, which has mathematical and practical interests. 

\section{Weak Galerkin Finite Element Schemes}
Let ${\cal T}_h$ be a partition of the domain $\Omega$ consisting of polygons in 2D. Suppose that  ${\cal T}_h$ is shape regular in the sense defined by (2.7)-(2.8) and each $T$ is convex. Let  ${\cal E}_h$ be the set of all edges in ${\cal T}_h$, and let ${\cal E}_h^{0}={\cal E}_h\setminus \partial \Omega $ be the set of all interior edges. For each element $T\in {\cal T}_h$, denote by $h_{T}$ the diameter of $T$, and $h=max_{T\in{\cal T}_h}h_{T}$ the mesh size of ${\cal T}_h$.

On $T$, let $P_{k}(T)$ be the space of all polynomials with degree no greater than $k$. Let  $V_{h}$ be the weak Galerkin finite element space associated with $T\in{\cal T}_h$ defined as follows:
\begin{eqnarray}
V_h&=&\{v=\{v_0,v_b\}: v_0|_T\in P_k(T), v_b|_e \in P_k(e), e \in \partial T, T\in{\cal T}_h \},
\end{eqnarray}
where $k\geq 1$ is a given integer. In this instance, the component $v_{0}$ symbolizes the interior value of $v$, and the component $v_{b}$ symbolizes the edge value of $v$  on each $T$ and $e$, respectively. Let $V^{0}_{h}$ be the subspace of $V_{h}$ defined as:
\begin{eqnarray}
	V_{h}^{0}&=&\{v: v\in V_{h},v_{b}=0\mbox{ on }\partial\Omega\}.
\end{eqnarray}

For any $v=\{v_{0},v_{b}\}\in V_{h}+H^{1}(\Omega),$ the weak gradient $\nabla_wv\in [P_{j}(T)]^{2}$ is defined on $T$ as the unique polynomial satisfying
\begin{eqnarray}
(\nabla_wv,\mathbf{q})_T&=&-(v_0,\mathbf{\nabla\cdot q})_T+\langle v_b,\mathbf{q\cdot \mathbf n}\rangle_{\partial T}, \quad \forall \mathbf{q}\in [P_{j}(T)]^2,j\geq k.\label{EQ:WeakGradient}
\end{eqnarray}
where $\mathbf{n}$ is the unit outward normal vector of $\partial T$.

For simplicity, we adopt the following notations,
\begin{equation}
(v,w)_{{\cal T}_h}=\sum_{T\in{\cal T}_h}(v,w)_{T}=\sum_{T\in{\cal T}_h}\int_{T}v w dx,\nonumber
\end{equation}
\begin{equation}
\left\langle v,w\right\rangle _{\partial{\cal T}_h}=\sum_{T\in{\cal T}_h}\left\langle v,w\right\rangle _{\partial T}=\sum_{T\in{\cal T}_h}\int_{\partial T}v w dx.\nonumber
\end{equation}

For any  $v=\left\lbrace v_{0},v_{b}\right\rbrace$  and $w=\left\lbrace w_{0},w_{b}\right\rbrace$ in $V_{h}$, we define the bilinear forms as follows:
\begin{eqnarray}
A(v,w)&=\sum_{T\in {\cal T}_{h}}(\nabla_w v,\nabla_w w)_{T}.\label{bilinear}
\end{eqnarray}

\begin{algorithm}A numerical solution for (\ref{bvp})-(\ref{dbc}) can be obtain by finding $u_{h}=\{v_{0},v_{b}\}\in V_{h}^{0}$, such that the following equation holds
\begin{eqnarray}
A(u_{h},v)=(f,v_{0}),\quad \forall v=\{v_{0},v_{b}\}\in V_{h}^{0}.\label{WG form}
\end{eqnarray}
where $A(\cdot,\cdot)$ is defined in (\ref{bilinear}).
\end{algorithm}

Accordingly, for any $v\in V_{h}^{0}$, we define an energy norm $\3bar \cdot \3bar$ on $V_{h}^{0}$ as:
\begin{equation}
\3bar v \3bar^2=\sum_{T\in{\cal T}_h}(\nabla_wv,\nabla_wv)_T=\sum_{T\in{\cal T}_h}\|\nabla_wv\|^{2}_{T}.\label{def-norm}
\end{equation}
An $H^{1}$ norm on $V_{h}^{0}$ is defined as:
\begin{equation*}
\|v\|_{1,h}^{2}=\sum_{T\in {\cal T}_{h}}\left( \|\nabla v_{0}\|^{2}_{T}+h_{T}^{-1}\|v_{0}-v_{b}\|^{2}_{\partial T}\right) .
\end{equation*}

The following well known lemma and corollary will be needed in proving our main result.
\begin{lemma}\label{lemyo}
Suppose that $D_{1}\subset D_{2}$ are convex regions in $\mathbb{R}^d$ such that $D_{2}$ can be obtain by scaling $D_{1}$ with a factor $r>1$. Then for any $p\in P_{n}(D_{2})$,
\begin{eqnarray*}
\|p\|_{D_{1}}\leq \|p\|_{D_{2}}\leq C r^{n} \|p\|_{D_{1}},
\end{eqnarray*}
where $C$ depends only on $d\mbox{ and } n$.
\end{lemma}
\begin{Corollary}\label{corol}
	Suppose that $D_{1}\mbox{ and } D_{2}$ satisfy the conditions of Lemma \ref{lemyo} and $D_{1}\subseteq D\subseteq D_{2}$. Then
	\begin{eqnarray*}
	\|p\|_{D_{1}}\leq\|p\|_{D}\leq Cr^{n}\|p\|_{D_{1}},\qquad\forall p\in P_{n}(D),
	\end{eqnarray*}
where $C$ depends only on $d\mbox{ and } n$.
\end{Corollary}

The following Theorem shows that for certain $j_{0}>0$, whenever $j\geq j_{0}$, $\|\cdot\|_{1,h}$ is equivalent to the $\3bar\cdot\3bar$ defined in (\ref{def-norm}), which is crucial in establishing the feasibility of SFWG. Later on, we will show that $j_{0}$ is optimal.
\begin{theorem}(cf.\cite{b2}, Lemma 3.1)
Suppose that $\forall T\in {\cal T}_{h}, T$ is convex with at most $m$ edges and satisfies the following regularity conditions: for all edges $e_{t}$ and $e_{s}$ of $T$
\begin{eqnarray}
|e_{s}|<\alpha_{0}|e_{t}|;\label{bnm}
\end{eqnarray}
for any two adjacent edges $e_{t} \mbox{ and } e_{s}$ the angle $\theta$ between them satisfies
\begin{eqnarray}
\theta_{0}<\theta<\pi-\theta_{0},\label{bnm1}
\end{eqnarray}
where $1\leq\alpha_{0}\mbox{ and } \theta_{0}>0$ are independent of $T$ and $h$. Let $j_{0}=k+m-2 \mbox{ or } j_{0}=k+m-3$ when each  edge of $T$ is parallel to another edge of $T$. When $deg\nabla_wv=j\geq j_{0}$, then there exist two constants $C_{1},C_{2}>0$, such that for each $v=\{v_{0},v_{b}\}\in V_{h}$, the following hold true
\begin{eqnarray*}
C_{1}\|v\|_{1,h}\leq \3barv\3bar\leq C_{2}\|v\|_{1,h},
\end{eqnarray*}
where $C_{1} \mbox{ and } C_{2}$ depend only on $\alpha_{0}\mbox{ and }\theta_{0}$.
\end{theorem}
\begin{proof}
For simplicity, from now on, all constant are independent of $T$ and $h$, unless otherwise mentioned. They may depend on $k,\alpha_{0},\mbox{ and } \theta_{0}$. Suppose $\nabla_w v\in [P_{j}(T)]^{2}, v\in P_{k}(T)$. We know that
\begin{eqnarray}
(\nabla_w v,\vec{q})_{T}=(\nabla v,\vec{q})_{T}+\left\langle v_{b}-v_{0},\vec{q}\cdot \vec{n}\right\rangle_{\partial T} \qquad\forall\vec{q}\in[P_{j}(T)]^{2}.\label{00}
\end{eqnarray}
Suppose $\partial T=(\cup_{i=0}^{m-1}e_{i})$, and we want to construct a $\vec{q}\in [P_{j_{0}}(T)]^{2}\subseteq[P_{j}(T)]^{2}$, where $j_{0}\leq j$, so that
\begin{eqnarray}
(\vec{q},p)_{T}&=&0,\quad \forall p\in[P_{k-1}(T)]^{2}\label{ewq1},\\\vec{q}\cdot\vec{n}_{i}&=&0 \mbox{ on } e_{i},i=1,\dots,m-1,\label{ewq11}\\\left\langle v_{b}-v_{0}-\vec{q}\cdot n_{0},p\right\rangle _{e_{0}}&=&0,\quad p\in P_{k}(e_{0}),\label{ewq111}
\end{eqnarray}
where $\vec{n}_{i}$ is the unit outer normal to $e_{i}$. Without loss of generality, we may assume that $e_{0}=\{(x,0)|x\in (0,1)\},\forall (x,y)\in T,y>0$, and one of the vertices of $e_{1}$ is $(0,0)$. Let $\ell_{i}\in P_{1}(T),i=2,\dots,m-1$ be such that $\ell_{i}(e_{i})=0,\ell_{i}(p)\geq0, \forall p\in T, \mbox{ and } \max_{p\in T}\ell_{i}(p)=1$. Let
\begin{eqnarray}
\vec{q}=\ell_{2}\dots\ell_{m-1}Q_{1}\vec{t}=LQ_{1}\vec{t}=Q\vec{t},\label{sq1}
\end{eqnarray}
where $\vec{t}$ is a unit tangent vector to $e_{1}$ and $Q_{1}\in P_{j_{0}-m+2}(T)$. It is easy to see that $\vec{q}\cdot\vec{n}_{i}|_{e_{i}}=0, i=1,\dots,m-1$. We want to find $Q_{1}$ so that
\begin{eqnarray}
\left\langle v_{b}-v_{0}-\ell_{2}\dots\ell_{m}
Q_{1}(\vec{t}_{1}\cdot\vec{n}_{0}),p\right\rangle _{e_{0}}&=&0,\quad\forall p\in P_{j_{0}-m+2}(e_{0})\label{210},\\(\ell_{2}\dots\ell_{m}Q_{1},p)_{T}&=&0,\qquad \forall p\in P_{k-1}(T)\label{211}.
\end{eqnarray}
Let $j_{0}-m+2=k$ or  $j_{0}=k+m-2$. It is easy to see that if $Q_{1}$ satisfies (\ref{210})-(\ref{211}), then $\vec{q}$ satisfies (\ref{ewq1})-(\ref{ewq111}). In addition, (\ref{210})-(\ref{211}) has $(k+1)(k+2)/2$ equations and the same number of unknowns. To show (\ref{210})-(\ref{211}) has a unique solution $Q_{1}\in P_{k}(T)$ we set $v_{b}-v_{0}=0$. Setting $p=Q_{1}$ in (\ref{210}) yields
\begin{eqnarray*}
\left\langle \ell_{2}\dots\ell_{m-1}Q_{1},Q_{1}\right\rangle _{e_{0}}=0.
\end{eqnarray*}
Since $L=\ell_{2}\dots\ell_{m-1}>0$ inside $e_{0}$, $Q_{1}(x,0)\equiv0$. Thus $Q_{1}=yQ_{2}$, where $Q_{2}\in P_{k-1}(T)$. Similarly
\begin{eqnarray*}
(y\ell_{2}\dots\ell_{m-1}Q_{2},p)_{T}=0\quad\forall p\in P_{k-1}(T),
\end{eqnarray*}
implies $Q_{2}=0$ and thus $Q_{1}=0$. Note that for any $j\geq j_{0}$, since $\vec{q}\in[P_{j_{0}}(T)]^{2}\subset[P_{j}(T)]^{2}$, there exists at least one such $\vec{q}$ in $[P_{j}(T)]^{2}$. Note also that when every edge of $\partial T$ is parallel to another one, we can lower the number of $\ell_{i}$ in (\ref{sq1}) by one. Thus for such a $T$, we have $j_{0}=m+k-3$. Next, we want to show that the unique solution of (\ref{ewq1})-(\ref{ewq111}) satisfies
\begin{eqnarray}
\|\vec{q}\|_{T}\leq \gamma_{0}\|\vec{q}\cdot\vec{n_{0}}\|_{e_{0}},\label{qqqqqq}
\end{eqnarray}
for some $\gamma_{0}>0$. Note that
\begin{eqnarray*}
(\vec{q}\cdot\vec{n_{0}})^{2}=Q^{2}\sin^{2}\theta=\vec{q}\cdot\vec{q}\sin^{2}\theta\geq\vec{q}\cdot\vec{q}\sin^{2}\theta_{0},
\end{eqnarray*}
where $\theta$ is the angle between $e_{0}$ and $e_{1}$. Thus, (\ref{qqqqqq}) is equivalent to
\begin{eqnarray}
\|Q\|_{T}\leq\gamma_{0}\|Q\|_{e_{0}},\label{xzx}
\end{eqnarray}
for some $\gamma_{0}>0$ and we may assume
\begin{eqnarray*}
e_{1}=\{(0,y)|0\leq y\leq 1\},
\end{eqnarray*}
because of the assumption (\ref{bnm}). Now let $T_{r}$ be the triangle spanned by $e_{0}$ and $e_{1}$ and label the third edge of $T_{r}$ as $e_{t}$. Let $\hat{T}_{r}\subset T_{r}$ be the triangle with vertices $(0,0.25),(0,0.75)$ and $(0.5,0.25)$. Label the edge of $\hat{T}_{r}$ connecting $(0,0.75)$ and $(0.5,0.25)$ as $\hat{e}_{t}$. For $i=2,\dots,m-1$, let
\begin{eqnarray*}
d_{i}&=&dist(\hat{T}_{r},L_{i}),\\D_{i}&=&\max_{p\in T}dist(p,L_{i}),
\end{eqnarray*}
where $L_{i}$ is the line containing $e_{i}$. Then
\begin{eqnarray}
d_{i}&\geq& dist(\hat{e}_{t},e_{t})=\sqrt{2}/4,\nonumber\\D_{i}&\leq& (|e_{0}|+|e_{1}|+\dots+|e_{m-1}|)/2\nonumber\\&\leq& \frac{m\alpha_{0}}{2}\label{12qw}.
\end{eqnarray}
Thus $\forall p\in \hat{T}_{r}$,
\begin{eqnarray}
\ell_{i}(p)\geq\frac{\sqrt{2}}{2m\alpha_{0}}=\varepsilon,i=2,\dots,m-1.\label{yuh1}
\end{eqnarray}
 Now we will prove (\ref{xzx}) in $3$ steps:
\begin{eqnarray*}
I:\beta_{1}\|Q_{1}\|_{e_{0}}&\leq&\|Q\|_{e_{0}},\\II:\beta_{2}\|Q\|_{T_{r}}&\leq&\|Q_{1}\|_{e_{0}},\\III:\beta_{3}\|Q\|_{T}&\leq&\|Q\|_{T_{r}}.
\end{eqnarray*}
\begin{enumerate}[ I.]
	\item Let $\hat{e}=\{(x,0)|0\leq x\leq 0.75\}$. Then similar to (\ref{yuh1}), $\forall p\in \hat{e}, \ell_{i}(p)\geq \varepsilon,i=2.\dots,m-1$. Thus
\begin{eqnarray*}
\varepsilon^{m-2}\|Q_{1}\|_{\hat{e}}\leq\|Q\|_{\hat{e}}\leq \|Q\|_{e_{0}}.
\end{eqnarray*}
It follows from Corollary \ref{corol} that
\begin{eqnarray}
\|Q\|_{e_{0}}\geq\varepsilon^{m-2}\|Q_{1}\|_{\hat{e}}\geq C\varepsilon^{m-2}\|Q_{1}\|_{e_{0}}=\beta_{1}\|Q_{1}\|_{e_{0}}.
\end{eqnarray}
	\item It is easy to see that
\begin{eqnarray*}
\|Q\|^{2}_{T_{r}}&=&\int_{0}^{1}\int_{0}^{1-y}L^{2}Q^{2}_{1} dx dy=\int_{0}^{1}\int_{0}^{1-y}L^{2}(Q_{1}-Q_{1}(x,0)+Q_{1}(x,0))^{2}dx dy\\&\leq&2\int_{0}^{1}\int_{0}^{1-y}\left[  L^{2}(Q_{1}-Q_{1}(x,0))^{2}+L^{2}\left(Q_{1}(x,0)\right) ^{2}\right]  dx dy.
\end{eqnarray*}
It follows from (\ref{yuh1}) and Corollary \ref{corol} that
\begin{eqnarray}
\|LQ_{1}(x,0)\|^{2}_{T_{r}}&=&\int_{0}^{1}\int_{0}^{1-y} L^{2}\left(Q_{1}(x,0)\right)^{2}dx dy\nonumber\\&\leq& \int_{0}^{1}\int_{0}^{1}\left( Q_{1}(x,0)\right)^{2} dxdy\nonumber\\&=&\|Q_{1}\|^{2}_{e_{0}}\label{rtgf}
\end{eqnarray}
	Note that
	\begin{eqnarray*}
		Q_{1}-Q_{1}(x,0)=y\bar{Q},\quad \bar{Q}\in P_{k-1},
	\end{eqnarray*}
	and
	\begin{eqnarray}
		\int_{T}\left[LQ_{1}(x,0)+Ly\bar{Q}\right] P dA=0,\quad\forall p\in P_{k-1}.\label{23wdst}
	\end{eqnarray}
Since $Ly\geq\dfrac{1}{4}\varepsilon \mbox{ for }(x,y)\in\hat{T}_{r}\mbox{ and } Ly\geq 0, \forall (x,y)\in T$,
\begin{eqnarray}
\int_{T}Ly\bar{Q}^{2} dA\geq\frac{\varepsilon}{4}\int_{\hat{T}_{r}}\bar{Q}^{2} dA=\delta_{3}\|\bar{Q}\|^{2}_{\hat{T}_{r}}.\label{aaaa}
\end{eqnarray}
Since $\|\cdot\|_{\hat{T}_{r}} \mbox{ and } \|\cdot\|_{T_{r}}$ are equivalent norms on $P_{k-1}(T_{r})$,
\begin{eqnarray}
\delta_{3}\|\bar{Q}\|_{\hat{T}_{r}}\geq\beta_{3}\|\bar{Q}\|_{T_{r}}.\label{aaaa1}
\end{eqnarray}
Let $T^{c}_{e}$ be the circumscribed isosceles-right triangle obtained by scaling $T_{r}$. It is easy to see that
\begin{eqnarray}
diam({T^{c}_{e_{0}}})\leq 2m\alpha_{0}.\label{plmn}
\end{eqnarray}
Since $T_{r}\subseteq T\subseteq T^{c}_{e}$, it follows from (\ref{aaaa}), (\ref{aaaa1}), and Corollary \ref{corol} that
\begin{eqnarray*}
\int_{T}Ly\bar{Q}^{2} dA\geq\beta_{3}\|\bar{Q}-{1}\|_{T_{r}}^{2}\geq\delta_{4}\|\bar{Q}\|^{2}_{T},
\end{eqnarray*}
where $\delta_{4}$ depend on (\ref{plmn}) and $k$. Letting $p=\bar{Q}$ in (\ref{23wdst}) yields
\begin{equation*}
\int_{T}Ly\bar{Q}^{2} dA=-\int_{T}LQ_{1}(x,0)\bar{Q} dA\leq \|Q_{1}(x,0)\|_{T}\|\bar{Q}\|_{T},
\end{equation*}
and thus
\begin{eqnarray*}
\|\bar{Q}\|_{T}\leq \frac{1}{\delta_{4}}\|Q_{1}(x,0)\|_{T}.
\end{eqnarray*}
Then
\begin{eqnarray*}
	\|\bar{Q}\|^{2}_{T}\leq \delta_{5}\int_{0}^{N}\int_{0}^{M}Q_{1}^{2}(x,0)dxdy=CN\int_{0}^{M}Q_{1}^{2}(x,0)dx=C\|Q_{1}\|^{2}_{[0,M]},
\end{eqnarray*}
where $N=\max_{(x,y)\in T}y\leq m\alpha_{0} \mbox{ and }M=\max_{(x,y)\in T}x\varepsilon\leq m\alpha_{0}$. It follows from Corollary \ref{corol} and (\ref{aaaa})
\begin{eqnarray}
\|Ly\bar{Q}\|_{\hat{T_{r}}}\leq\|\bar{Q}\|_{T}\leq C\|Q_{1}\|_{e_{0}}.\label{qwer}
\end{eqnarray}
Combining (\ref{qwer}) and (\ref{rtgf}) yields
\begin{eqnarray}
\|Q\|_{T_{r}}\leq C\|Q_{1}\|_{e_{0}}.
\end{eqnarray}
	\item Applying Corollary \ref{corol} again yields
	\begin{eqnarray}
	C\|Q\|_{T}\leq \|Q\|_{T_{r}}.
	\end{eqnarray}
	Now we have completed the $3$-step argument.\newline
	 By a scaling argument, we have
	\begin{eqnarray}
	\|Q\|_{T}\leq \gamma_{0}h^{1/2}\|Q\|{e_{0}}.\label{27uij}
	\end{eqnarray}
	Plugging $\vec{q}$ into (\ref{00}) yields
	\begin{eqnarray}
	\|Q\|_{T}\|\nabla_wv\|_{T}&\geq&(\nabla_wv,\vec{q}\vec{n})_{\hat{T}}=\left\langle v_{b}-v_{0},\vec{q}\vec{n}\cdot\vec{n}\right\rangle _{e_{0}}\nonumber\\&=&\|v_{b}-v_{0}\|_{e_{0}}^{2}\nonumber\\&=&\|v_{b}-v_{0}\|_{e_{0}}\|Q\|_{e_{0}}.
	\end{eqnarray}
	It follows from (\ref{27uij}) that
	\begin{eqnarray*}
	\|\nabla_wv\|_{T}\geq Ch^{1/2}\|v_{b}-v_{0}\|_{e_{0}}.
	\end{eqnarray*}
\end{enumerate}
Similarly,
\begin{eqnarray*}
	\|v_{b}-v_{0}\|_{e_{i}}\leq Ch^{1/2}_{T}\|\nabla_wv\|_{T},i=1,\ldots,m-1.
\end{eqnarray*}
Thus
\begin{eqnarray*}
 \|v_{b}-v_{0}\|_{\partial T}\leq Ch^{1/2}_{T}\|\nabla_wv\|_{T}.
\end{eqnarray*}
By letting $\mathbf{q}=\nabla_wv$ in (\ref{00}), we arrive at
\begin{eqnarray*}
(\nabla_w v,\nabla_wv)_{T}=(\nabla v_{0},\nabla_wv)_{T}+\left\langle v_{b}-v_{0},\nabla_wv\cdot \mathbf{n}\right\rangle_{\partial T}.
\end{eqnarray*}
It follows from the trace inequality and the inverse inequality that
\begin{eqnarray*}
\|\nabla_wv\|^{2}_{T}&\leq& \|\nabla v_{0}\|_{T}\|\nabla_wv\|_{T}+\|v_{0}-v_{b}\|_{\partial T}\|\nabla_wv\|_{\partial T}\\&\leq&\|\nabla v_{0}\|_{T}\|\nabla_wv\|_{T}+Ch_{T}^{1/2}\|v_{0}-v_{b}\|_{\partial T}\|\nabla_wv\|_{T}.
\end{eqnarray*}
Thus
\begin{eqnarray*}
\|\nabla_wv\|_{T}&\leq&C\left( \|\nabla v_{0}\|_{T}+h_{T}^{1/2}\|v_{0}-v_{b}\|_{\partial T}\right).
\end{eqnarray*}
\end{proof}

To show that $j_{0}=k+m-2 \mbox{ or }j_{0}=k+m-3$ is the lowest possible degree for the weak gradient, we give the following examples.
\begin{example}\label{exas1}
	Suppose $\Omega=(0,1)\times(0,1)$ and is partitioned  as shown in Figure \ref{fi11}.
	\begin{figure}[h!]
		\centering
		\includegraphics[width=0.7\textwidth]{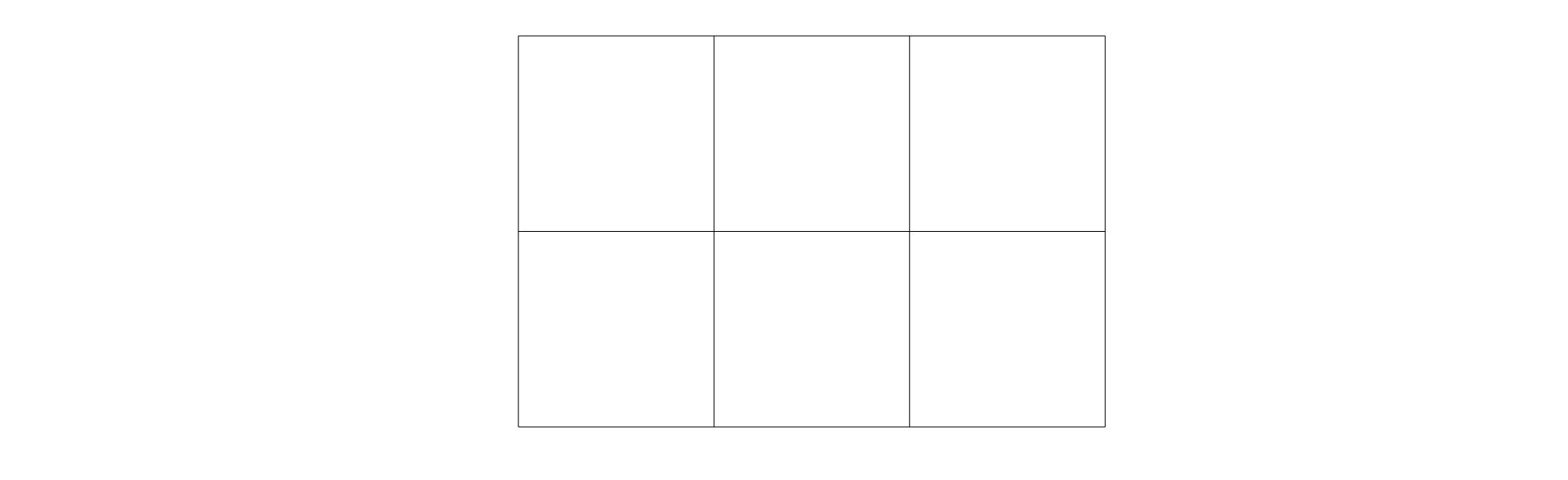}
		\caption{}\label{fi11}
	\end{figure}\\
	Let $j=k=1=j_{0}-1$. To see if (\ref{WG form}) has a unique solution in $V_{h}^{0}$, we set $f=0$ and $v=u_{h}$. Then $\nabla_w u_{h}=0$. By definition,
	\begin{eqnarray*}
	0=(\nabla_w u_{h},\vec{q})_{{\cal T}_{h}}=-(u_{h},\nabla \vec{q})_{{\cal T}_{h}}+\left\langle u_{b},\vec{q}\cdot\vec{n} \right\rangle_{\partial {{\cal T}_{h}}},\quad \forall\vec{q}\in[P_{1}({\cal T}_{h})]^{2}.
	\end{eqnarray*}
Note that the degree of freedom of $[P_{1}({\cal T}_{h})]^{2}$ is $36$ while the degree of freedom of $V_{h}^{0}$ is $39$. Thus, there exist $u_{h}\neq0$ so that $\nabla_w u_{h}=0$. Thus Algorithm $1$ will not work.
\end{example}
\begin{example}
	Let $\Omega=(0,1)\times (0,1)$ and is partitioned as shown in Figure \ref{yhb}.
	\begin{figure}[h!]
		\centering
		\includegraphics[width=0.7\textwidth]{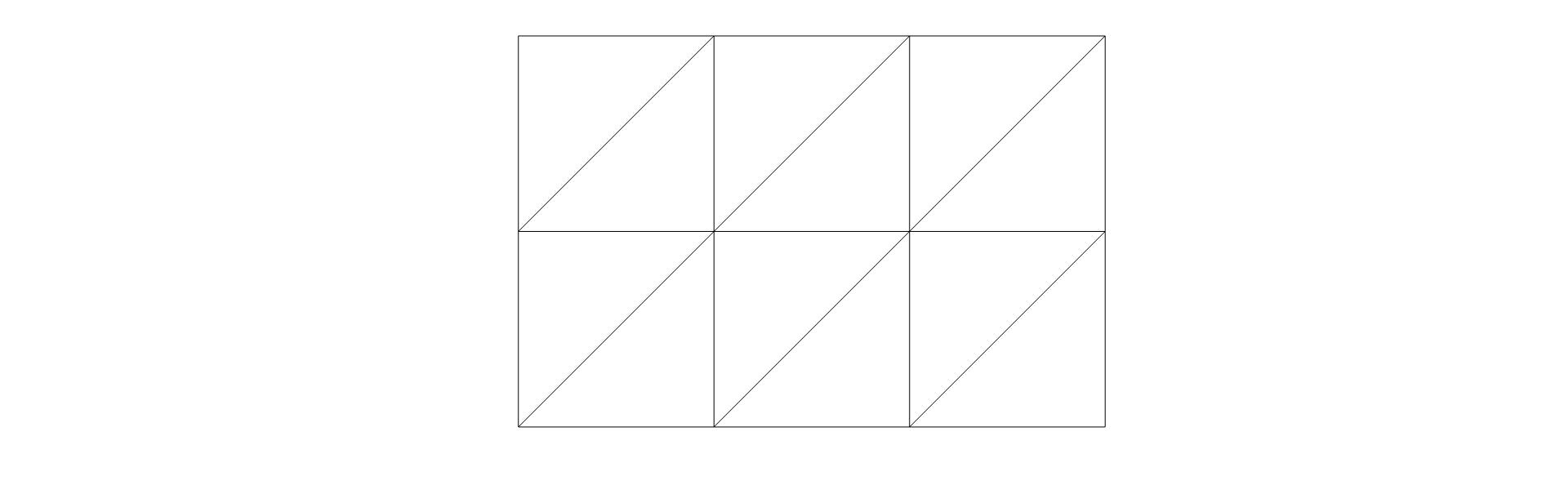}
		\caption{}\label{yhb}
	\end{figure}
	Let $j=k=1=j_{0}-1$.
	Note that the degree of freedom of $[P_{1}({\cal T}_{h})]^{2}$ is $72$ while the degree of freedom of $V_{h}^{0}$ is $75$. A similar argument as in example \ref{exas1} shows that Algorithm $1$ will not work.
\end{example}\\

For the completeness, we list the following from \cite{b2}.
\begin{theorem}
	Let $u_{h}\in V_{h}$ be the weak Galerkin finite element solution of (\ref{WG form}). Assume that the exact solution $u\in H^{k+1}(\Omega)$ and $H^{2}$-regularity holds true. Then, there exists a constant $C$ such that
	\begin{eqnarray}
		\3baru-u_{h}\3bar&\leq& Ch^{k}|u|_{k+1},\\\|u-u_{0}\|&\leq& Ch^{k+1}|u|_{k+1}.
	\end{eqnarray}
\end{theorem}
\section{Numerical Experiments}
We are devoting this section to verify our theoretical results in previous sections by two numerical examples. The domain in all examples is $\Omega=(0,1)\times (0,1)$. We will implement the SFWG finite element method (\ref{WG form}) on triangular meshes.
A uniform triangulation of the domain $\Omega$ is used. By refining each triangle for $N=2, 4, 8, 16, 32,64,128$, we obtain a sequence of partitions. The first two levels of grids are shown in Figure \ref{fi3}. We apply the SF-WG finite element method with $(P_{k}(T),P_{k}(e),[P_{k+1}(T)]^{2}),k=1,2$ finite element space to find SFWG solution $u_h=\{u_0,u_b\}$ in the computation.
\begin{figure}[h!]
	\centering
	\includegraphics[width=0.6\textwidth]{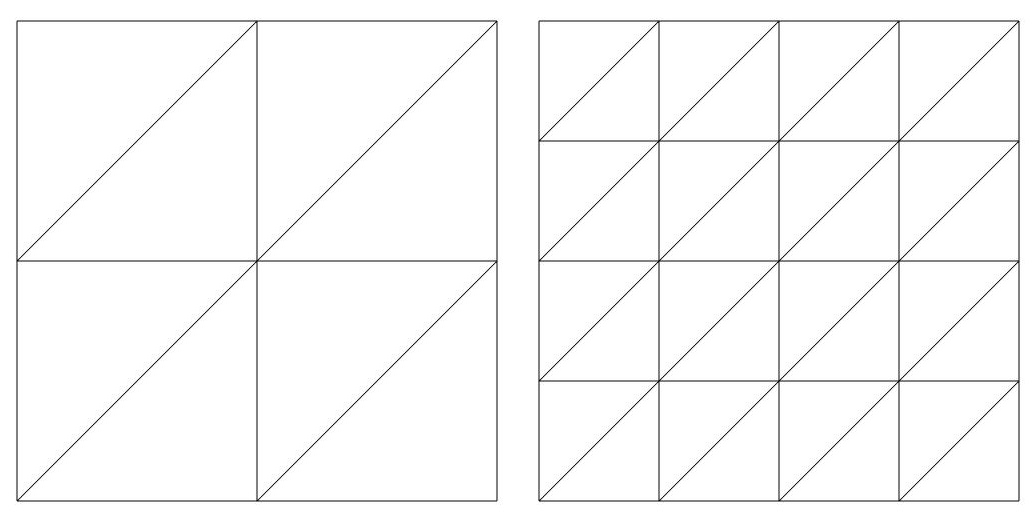}
	\caption{The first two triangle grids in computation}\label{fi3}
\end{figure}
\begin{example}
In this example, we solve the Poisson problem (\ref{bvp})-(\ref{dbc}) posed on the unit square $\Omega$, and  the analytic solution is
\begin{eqnarray}
u(x,y)=\sin(\pi x)\sin(\pi y).
\end{eqnarray}
The boundary conditions and the source term $f(x,y)$ are computed accordingly. Table \ref{tab:example1} lists errors and convergence rates in $H^{1}$-norm and $L^{2}$-norm.
\begin{table}[H]
	\caption{Error profiles and convergence rates for $P_{k}$ elements with $P_{k+1}^{2}$ weak gradient, $(k=1,2)$.}
	\label{tab:example1}
	\small 
	\centering 
	\begin{tabular}{lccccr} 
		\toprule[\heavyrulewidth]\toprule[\heavyrulewidth]
	\textbf{$k$} &	\textbf{$h$} & \textbf{$\3baru_{h}-Q_{h}u\3bar$} & \textbf{$\mbox{Rate}$} & \textbf{$\|u_{0}-Q_{0}u\|$} & \textbf{$\mbox{Rate}$} \\
		\midrule
		&1/2   &8.8205E-01  & -     &8.0081E-02  & - \\
		&1/4   &4.3566E-01  & 1.02  &2.5055E-02  & 1.68\\
	  	&1/8   &2.1565E-01  & 1.01  &6.7091E-03  & 1.90\\
	  1	&1/16  &1.0747E-01  & 1.00  &1.7092E-03  & 1.97\\
		&1/32  &5.3684E-02  & 1.00  &4.2940E-04  & 1.99\\
		&1/64  &2.6836E-02  & 1.00  &1.0749E-04  & 2.00\\
		&1/128 &1.3417E-02  & 1.00  &2.6871E-05    & 2.00\\
	    \bottomrule[\heavyrulewidth]\\
		&1/2  &2.5893E-01     &- &1.0511E-02 &- \\
		&1/4  &6.5338E-02 & 1.99 &1.2889E-03 &3.03\\
	  	&1/8  &1.6252E-02 & 2.00 &1.5596E-04 &3.05\\
	  2	&1/16 &4.0538E-03 & 2.09 &1.9193E-05 &3.02 \\
		&1/32 &1.0129E-03 & 2.00 &2.3837E-06 &3.00 \\
		&1/64 &2.5321E-04 & 2.00 &2.9713E-07 &3.00\\
        &1/128 &6.3304E-05  & 2.00 &3.7094E-08 &3.00\\
      	\bottomrule[\heavyrulewidth]
	\end{tabular}
\end{table}
\end{example}
\begin{example}
In this example, we consider the problem $-\Delta u=f$ with boundary condition and the exact solution is
\begin{eqnarray}
u(x,y)=x(1-x)y(1-y).
\end{eqnarray}
The results obtained in Table \ref{tab:example2} show  the errors of SFWG scheme and the convergence rates in the $\3bar\cdot\3bar$ norm  and $\|\cdot\|$ norm. The simulations
are conducted on triangular meshes and polynomials of order $k = 1, 2$. The SFWG scheme with $P_{k}$ element has optimal convergence rate of $O(h^{k})$ in $H^{1}$-norm and $O(h^{k+1})$ in $L^{2}$-norm.
\begin{table}[H]
	\caption{Error profiles and convergence rates for $P_{k}$ elements with $P_{k+1}^{2}$ weak gradient, $(k=1,2)$.}
	\label{tab:example2}
	\small 
	\centering 
	\begin{tabular}{lccccr} 
		\toprule[\heavyrulewidth]\toprule[\heavyrulewidth]
		\textbf{$k$}&\textbf{$N$} & \textbf{$\3bar e_{h}\3bar$} & \textbf{$\mbox{Rate}$} & \textbf{$\|e_{0}\|$} & \textbf{$\mbox{Rate}$} \\
    	\midrule
	   &1/2   &5.6636E-02   & -     &5.3064E-03  & - \\
	    &1/4   &2.9670E-02  & 0.93  &1.6876E-03   & 1.65\\
	    &1/8   &1.4948E-02  & 0.99  &4.5066E-04    & 1.90\\
	1	&1/16  &7.4848E-03  & 1.00  &1.1461E-04    & 1.98\\
	    &1/32  &3.7435E-03  & 1.00  &2.8778E-05      & 1.99\\
	    &1/64  &1.8719E-03  & 1.00  &7.2024E-06      & 2.00\\
	    &1/128 &9.3596E-04  & 1.00  &1.8011E-06      & 2.00\\
	    \bottomrule[\heavyrulewidth]\\
	    &1/2  &1.3353E-02 &-     &5.6066E-04  & - \\
	    &1/4  &3.4831E-03& 1.94 &6.8552E-05  & 3.03\\
	    &1/8  &8.7712E-04 & 1.99 &8.2466E-06  & 3.06\\
	2	&1/16 &2.1961E-04 & 2.00 &1.0094E-06  & 3.03\\
	    &1/32 &5.4967E-05 & 2.00 &1.2482E-07  & 3.02 \\
	    &1/64 &1.3747E-05 & 2.00 &1.5518E-08  & 3.00\\
	    &1/128 &3.4375E-06& 2.00 &1.9344E-09  & 3.00\\

	\bottomrule[\heavyrulewidth]
	\end{tabular}
\end{table}
\end{example}
Overall, the numerical experiment's results shown in the formentioned examples match the theoretical study part in the previous sections.

\end{document}